\documentclass[a4paper,12pt]{article}
\usepackage{amsmath,amssymb,amsfonts,amscd,amsthm}
\usepackage{graphicx}
\usepackage{dsfont}
\usepackage{times}
\usepackage{blindtext}
\usepackage[titletoc, toc, title, page]{appendix}
\usepackage{color}
\usepackage{enumerate}
\usepackage{fancyhdr}		
\usepackage{pgf,tikz}
\usepackage{sectsty} \definecolor{bleu1}{RGB}{0,57,128}
\usepackage[utf8]{inputenc}
\usepackage[english]{babel}

\newtheorem{Main}{\bg Theorem}

\def\tk{\tilde{k}}
\def\bk{\black}
\newtheorem{Coro}{\bg Corollary}

\usepackage[colorlinks,linkcolor=bleu1,anchorcolor=green,citecolor=bleu1]{hyperref}
\subsectionfont{\color{bleu1}}
\sectionfont{\color{bleu1}}
\subsectionfont{\color{bleu1}}

\usepackage{xcolor}
\definecolor{dgreen}{rgb}{0.1,0.6,0.1}

\definecolor{bluegray}{rgb}{0.4, 0.6, 0.8}

\def\bg{\color{bleu1}}
\def\bk{\color{black}}

\usepackage{xcolor}

\newtheorem{theorem}{\bg Theorem}                             
\newtheorem{lemma}{\bg Lemma}
\newtheorem{proposition}{\bg Proposition}

\newtheorem{remark}{\bg Remark}

\numberwithin{equation}{section}

\def\bg{\color{bleu1}}

\makeatletter
\def\author#1{\gdef\autrun{\def\and{\unskip, }#1}\gdef\@author{#1}}
\def\address#1{{\def\and{\\\hspace*{18pt}}\renewcommand{\thefootnote}{}%
\footnote {#1}}%
}

\makeatother
\def\email#1{e-mail: #1}

\def\keywords#1{\par\medskip
\noindent\textbf{Keywords.} #1}

\renewenvironment{proof}{\noindent {\textit{\bg Proof.}}}{\qed\vskip 0.2cm}

\makeatletter
\newenvironment{proofof}[1]{\par
  \pushQED{\qed}%
  \normalfont \topsep6\p@\@plus6\p@\relax
  \trivlist
  \item[\hskip\labelsep
    \textit{\bg Proof of #1\@addpunct{.}}]\ignorespaces
}{%
  \popQED\endtrivlist\@endpefalse
}
\makeatother

\newcommand{\etalchar}[1]{$^{#1}$}

\newcommand{\C}{\mathbb C}
\newcommand{\R}{\mathbb R}

\newcommand{\Z}{\mathbb Z}
\newcommand{\N}{\mathbb N}

\newcommand{\T}{\mathbb{T}}
\newcommand{\cT}{\mathcal{T}}

\newcommand{\im}{\mathop{\mbox{Im}}}

\renewcommand{\inf}{\mathop{\mbox{{inf}}}
	\renewcommand{\lim}{\mathop{\mbox{lim}}}
	\renewcommand{\max}{\mathop{\mbox{ma}x}}}

\providecommand{\abs}[1]{\lvert#1\rvert}	
\providecommand{\norm}[1]{\lVert#1\rVert}

\newcommand{\om}{\omega} 
\newcommand{\cO}{\mathcal O} 

\begin{document}


\baselineskip=17pt



\title{\bg Instabilities for analytic quasi-periodic invariant tori}

\author{Gerard Farré \and Bassam Fayad}

\date{}

\maketitle

\address{G. Farré: KTH, Department of Mathematics; \email{gerardfp@kth.se}
\and
B. Fayad: CNRS, IMJ-PRG; \email{bassam.fayad@imj.prg.fr}}


\begin{abstract}

We prove the existence  of real analytic Hamiltonians with topologically unstable quasi-periodic invariant tori. Using various versions of our examples, we solve the following problems in the stability theory of analytic quasi-periodic motion: 
\begin{itemize}
\item[$i)$] \bk Show the existence of topologically unstable  tori of arbitrary frequency. Moreover, the Birkhoff Normal Form at the invariant torus can be chosen to be convergent, equal to a planar or non-planar polynomial. 
\item[$ii)$] \bk Show  the optimality of the exponential stability for Diophantine tori.
\item[$iii)$] \bk Show the existence of real analytic Hamiltonians that are integrable on half of the phase space, and such that all orbits on the other half accumulate at infinity.
\item[$iv)$] \bk For sufficiently Liouville vectors, obtain invariant tori that are not accumulated by a positive measure set of quasi-periodic invariant tori.  

\end{itemize}

\end{abstract}
\section{\large{\bf{Introduction}}}

Let  $H$ be a $C^2$ function defined on $\T^d \times \R^d$ and consider its Hamiltonian vector field $X_{H}(\theta, r)=(\partial_{r}H(\theta,r), -\partial_{\theta}H(\theta, r))$. 
If for some $\omega \in \R^d$, we have
\begin{equation} \tag{{$*$}}
H(\theta,r)= \langle \omega,r\rangle+\cO(r^2) \label{HH},
\end{equation}
then $\cT_0=\T^d \times \{0\} $ is invariant under the Hamiltonian flow $\Phi^t_{H}$
and the induced dynamics on this torus  is the translation of frequency vector $\omega$
$: \theta \mapsto \theta+t\omega$. 
Moreover this torus is Lagrangian with respect to the 
canonical symplectic form $d \theta \wedge d r$ on $\T^d \times \R^d$. 

In this work, we will mainly be interested in the non-resonant case, where the coordinates of $\omega$ are rationally independent, in which case the torus $\cT_0$ can be seen as the closure of any orbit that starts on $\cT_0$. We call such an invariant torus a quasi-periodic torus of the Hamiltonian $H$, and for short, a QP torus.  
 
The study of the stability properties of a QP torus is an old  problem of classical mechanics,
especially in relation to the N-body problem of celestial mechanics. There exist three different notions of stability.   
The usual topological or Lyapunov stability, the stability in a measure theoretic or probabilistic sense (KAM stability),
and the effective stability or quantitative stability in time.

In this paper, we  will use variants of the approximation by conjugation method (AbC or Anosov-Katok method) to construct several examples with various instability properties of QP tori of a real analytic Hamiltonian, from all three  points of view and in relation with the main known results and open questions in the field.

We show in particular the existence  of real analytic Hamiltonians with topologically unstable quasi-periodic invariant tori with arbitrary frequencies. We also show sharpness of several results in Nekhoroshev and KAM theory.

In the AbC method, diffeomorphisms or flows of a manifold are constructed as limits of conjugates of diffeomorphisms or periodic flows. Volume preserving maps with various interesting, sometimes surprising, topological and ergodic properties can be obtained as limits of volume preserving periodic transformations that are conjugates {\it via} wild conjugacies to a simple periodic action on the manifold. For a general overview of the conjugation by approximation method we refer the reader to \cite{faka}.

In the Hamiltonian setting, by periodic approximations we mean that the flow in some of the angle variables will be approached by periodic flows, which causes instabilities and drift in the action coordinates. More precisely,  all our examples will be of the following form (see Section  \ref{sec_constructions} for complete details):
\begin{equation} 
\label{HaHaHaHa}
 H=\lim_{n\rightarrow \infty}{H_n}, \quad H_n(\theta,r)=\langle  \om(r_d),r \rangle -\sum_{j=2}^{n}\phi_j(r_d) \sin (2\pi {\small{ \sum_{i=1}^{d-1}{k_{j,i}\theta_i}} }).
\end{equation}
Here $\{k_j\}$ is a sequence of vectors in $\Z^{d-1}$, $\om(\cdot)$ is a function from $\R$ to $\R^d$, and the $\phi_j(\cdot)$ are positive functions, that will be chosen  adequately to guarantee the convergence of the sequence $H_n$ and the various properties of our examples. 
Note that $r_d$ is constant under the Hamiltonian flows we are considering and acts as a {\it parameter} in the constructions, in a way that we now explain. Note also that the flows are explicitly solvable, for $r_d$ near the origin, and are conjugated to  $H_0=\langle  \om(r_d),r \rangle$. Namely, there exists an explicit canonical transformation $\Psi_n$ such that $H_n= H_0 \circ \Psi_n$ (see Section \ref{sec.bnf} for the explicit form of $\Psi_n$). The wild behavior of the conjugacies $\Psi_n$ can be caused by two possible scenarios: 
\begin{itemize}
\item[i)] The frequency vector $\om(r_d)\equiv \tilde{\omega}$ is a constant Liouville vector (see Section \eqref{sec.notations} and \eqref{eq_liouville_cond} for the definitions) and the sequence $\{k_j\}$ is a sequence of almost resonant vectors with respect to the frequency $\tilde{\omega}$. 
\item[ii)] As $r_d$ goes to zero, the frequency vector $\om(r_d)$ goes through resonances  that correspond to the sequence $\{ k_j \}$.
\end{itemize}
The AbC method was first applied in the Hamiltonian context by Katok in \cite{K73}. Recent applications of the AbC method in the Hamiltonian context that inspired this work  can be found in \cite{EFK,FSpoint}.

\vspace{0.2cm}

\section{\large{\bf{Notations}}} \label{sec.notations}

Let us introduce some notations that will be useful throughout the paper.

\begin{itemize}

\item For any vector $v=(v_1, \ldots, v_d )\in \R^d$ we will denote $\norm{v}:=\max_{1 < m \leq d}{\abs{v_m}}$.  

\item We denote by $\T^d_{\rho}$ the complex $\rho$-neighbourhood of a 
standard real $d$ dimensional torus 
\begin{equation*}
 \T^d_{\rho}=\left\{z \in \C^d / \Z^d \; | \; \abs{\im z_i} < \rho, \; 1\leq i\leq d\right\}.
\end{equation*}
We denote by $B_{\Delta ,\rho}$ the complex $\rho$-neighbourhood of the closed ball 
$B_{\Delta} \subset \R^d$ centered at the origin with radius $\Delta>0$,
\begin{equation*}
 B_{\Delta,\rho}=\left\{z \in \C^d \; | \; \exists z'\in B_\Delta \; s.t \; \abs{z-z'} < \rho \right\}. 
\end{equation*}
 We will also denote $D_{\Delta, \rho}=  \T^{d}_{\rho} \times B_{\Delta, \rho}$.

\item A holomorphic function $f$ defined on $D_{\Delta,\rho}$ is said to be real if  it gives real values to real arguments. 
 We will  denote by $C^{\omega}_{\Delta,\rho}$ the real and bounded holomorphic functions $f:D_{\Delta, \rho} \rightarrow \C$, 
which form a Banach space with the supremum norm
\begin{equation*}
\norm{f}_{\Delta, \rho}=\sup_{z\in D_{\Delta,\rho}}{\abs{f(z)}}.
\end{equation*}
By $C^{\omega}_{0,\rho}$ we denote the subset of functions of $C^{\omega}_{\Delta,\rho}$ that depend only on $\theta$.  We will denote by $C^{\omega}$ the real holomorphic entire functions and $C^{\omega}_{\rho}:=\bigcap_{\Delta>0}{C^{\omega}_{\Delta, \rho}}$. 
Recall that with the compact-open topology both are a Fr\'echet spaces. 
In particular we will use that convergence in $C^{\omega}_{\Delta,\rho} \; \forall \Delta,\rho>0$ implies convergence in $C^{\omega}$, and that convergence in $C^{\omega}_{\Delta,\rho} \; \forall \Delta>0$ for a fixed $\rho>0$ implies convergence in $C^{\omega}_{\rho}$. 

\item {\it Formal power series.} Let $z=(z_1,\dots,z_d) \in \C^d$. An element 
$$f\in \mathcal C^{\omega}(\T^d_\rho)[[z]]$$
is a formal power series
$$f=f(\theta,z)=\sum_{j \in \N^d} a_j(\theta) z^j$$
whose coefficients $a_j\in C^{\omega}_{0,\rho}$ 
(possibly vector valued). 

\item Given a vector $v=(v_1, \ldots, v_d)\in \R^{d}$ we denote by
$\tilde{v}:=(v_1,\ldots, v_{d-1})\in \R^{d-1}$ the new vector obtained by omitting the last component.
Similarly for a map $f:\R\rightarrow \R^d$ we will denote by $\tilde{f}:\R \rightarrow \R^{d-1}$
the corresponding map where the last component is omitted.

\item We will usually denote the last component of $r=(r_1, \ldots, r_d)\in \R^{d}$ by $s:=r_d$ to distinguish it from the 
rest of the components. We do so to stress the fact that in our constructive methods $s$ plays the role of a parameter, it does not change with time. 
This happens because all the Hamiltonians we consider will not depend on $\theta_d$, and thus satisfy $\dot{s}=-\frac{\partial H}{\partial \theta_d}=0$ (see Section \ref{sec_constructions}).

\item We call $\omega$ a Diophantine vector of exponent $\tau> 0$ and constant $\gamma>0$ if
\begin{equation*}
\abs{ \langle \omega, k \rangle} \geq \frac{\gamma}{\norm{k}^{\tau}}, \quad  \forall \; k\in \Z^{d}\setminus \{0\}. 
\end{equation*}
We denote by  $\Omega^{d}_{\gamma, \tau} $ the set of all such vectors.
Recall that for any $\tau>d-1$, the set of all Diophantine vectors of exponent $\tau$ : $\Omega^{d}_{\tau}:=\bigcup_\gamma \Omega^{d}_{\gamma, \tau}$ has full Lebesgue measure. A non-resonant vector that is not Diophantine is called a Liouville vector. 

\end{itemize}
 \section{\bf A brief reminder on Birkhoff normal forms and KAM stability.}

{ \subsection{Birkhoff normal forms.} }

We say that $H$ as in \eqref{HH} has a normal form $N_H$, if   $N_H$ is a formal power series in $r$ (possibly with 0 radius of convergence) and there exists a formal power series 
$$f\in\mathcal C^{\omega}(\T_{\rho}^d)[[r]]\cap \cO^2(r)$$
such that
$$H(\theta, r+\partial_{\theta} f(\theta,r))=N_H(r).$$
If a normal form exists at a QP torus (non-resonant by our definition) it is unique. It is then called the Birkhoff normal form
of $H$ at the QP torus (we refer to \cite{Bi} or \cite{SM} for more details on Birkhoff normal forms). A classical result is that when $H$ is as in \eqref{HH} and $\omega$ is Diophantine, the normal form exists and is unique.

\subsection{Non-degenerate Birkhoff Normal Forms and KAM stability. } \bk 

A QP torus of a Hamiltonian system is said to be KAM stable if it is accumulated by a positive measure of QP tori,
and if the set of these tori has Lebesgue density one at the original torus.
 We say that a formal power series $N_H$ is {\it non-degenerate} or {\it non-planar} if there does not exist any vector  $\gamma$ such that for every   $r$ in some neighborhood of $\cT_0$, $\langle \nabla N_H(r),\gamma\rangle=0$. The following was proven in \cite{EFK}.

\begin{Main}\label{mB}
If $N_H$ exists, is unique  and  is non-degenerate,  then $\cT_0$ is KAM stable. 
 In particular, this is the case if $\omega$ is Diophantine and if $N_H$ is non degenerate.\end{Main}

The condition that $N_H$ is non-degenerate is essentially equivalent to R\"ussmann's non-degeneracy
condition that guarantees the survival of a QP torus of an integrable system under small perturbations (see \cite{russ,jiang}). In \cite{EFK}, it was shown to be a sufficient condition for KAM stability in the singular perturbation problem that appears in the study of the stability of a QP torus. 

\subsection{On the convergence of the BNF}  \label{sec.div} We know that a convergent symplectic coordinate change that yields the BNF exists if and only if  $H$ is integrable \cite{I} (see also \cite{V,N}). It was known to Poincar\'e that 
for ``typical'' (in a sense we would call today generic) $H$, $f_H$ will be divergent. Siegel 
\cite{S55} proved the same thing
in a neighborhood of an elliptic equilibrium with another, and stronger,
notion of ``typical''. However, this does not solve the question of the convergence of the  BNF itself, that is always defined when $\omega$ is Diophantine. When the radius of convergence of the formal power series $N_H(\cdot)$ is $0$, we say that the BNF diverges. 

For example,  the following questions were asked by Eliasson \cite{E1,eliasson,EFK}:
\begin{itemize}
 \item[($i)$] \bk \bk can $N_H$ be divergent?
 \item[($ii)$] \bk \bk if  $H$ is non integrable, can $N_H$ be convergent?
\end{itemize}
A result of Perez-Marco \cite{PM} states, for any fixed vector $\omega$, that if 
$N_H$ is divergent for some $H$ as in \eqref{HH}, then $N_H$ is divergent for 
``typical'' (i.e. except for a pluri-polar set) $H$. 

In \cite{Fpoint}, it was shown that for any $\om \in \R^{d}, d \geq 4$ such that $\omega_1 \omega_2 <0$
 there exists a real entire Hamiltonian $H: \R^{2d} \to \R$ such that the origin is an elliptic equilibrium with frequency $\om$ and such that the BNF of $H$ at the origin is divergent. 
  This construction can readily be extended to the case of QP tori as in \eqref{HH}. It follows from \cite{PM} that for any Diophantine $\omega \in \R^d$, $d\geq 4$, the BNF at a QP torus of frequency $\omega$ is generically divergent.

{\it A contrario}, one of the results that will be obtained here is an answer to ($ii)$ with an example of a real entire Hamiltonian as in \eqref{HH}, with arbitrary non-resonant frequency $\omega \in \R^3$, such that the BNF at $\cT_0$ exists and is convergent but $\cT_0$ is Lyapunov unstable and thus $H$ is non integrable.

Extending this result to elliptic fixed points is unfortunately not readily available because the action angle coordinates are singular at the origin, and the extension of real analytic unstable constructions in this direction (from tori to points) is a  challenging problem. For instance, it is not known how to adapt the Approximation by Conjugations construction  method on the disc to the real analytic category (see \cite{FKicm} for a discussion on this topic).

\section{Statement of the main results}


\subsection{Lyapunov stability} \bk A closed invariant set of an autonomous Hamiltonian flow  is said to be  Lyapunov stable or 
topologically stable if all nearby orbits remain close to it for all forward time. R. Douady gave in \cite{douady} examples of smooth Hamiltonians having a Lyapunov unstable QP torus. 
 Douady's examples can have any chosen Birkhoff Normal Form at the origin provided its Hessian at the fixed point is non-degenerate. Douady's examples are modelled on the Arnold diffusion mechanism through chains of heteroclinic intersections between lower dimensional partially hyperbolic invariant tori that accumulate towards the origin. The construction consists of a countable number of compactly supported perturbations of a completely integrable flow, and as such was carried out only in the $C^\infty$ category. Examples of smooth Hamiltonians having a Lyapunov unstable QP torus with a degenerate Birkhoff normal form were obtained in \cite{EFK,FSpoint}. 
 
 Topological instability of a QP torus is conjectured to hold for generic systems in $3$ or more degrees of freedom.
In fact, it was conjectured by Arnol'd that a ``general'' Hamiltonian should have a dense orbit on a ``general'' energy surface \cite{arnold}.
A great amount of work has been dedicated to proving
this conjecture (giving a precise meaning to the word ``general''), 
but the picture is not yet completely clear, especially when it comes to real analytic Hamiltonians (see for example \cite{BKZ16} and references therein).  For instance, not a single example was known up to now of a real analytic Hamiltonian that has a   Lyapunov unstable QP torus. It was shown in \cite{Fpoint} that for any $\om \in \R^{d}, d \geq 4$, such that not all its coordinates are of the same sign,  there exists a real entire Hamiltonian  such that the origin is a Lyapunov unstable  elliptic equilibrium with frequency $\om$. As we discussed earlier, the construction of \cite{Fpoint} can readily be extended to the case of QP tori and the condition on the sign of the coordinates of $\om$ can be dropped. 
However, all  the examples that one obtains following the method of \cite{Fpoint} would have a divergent BNF. 
  
The constructions in this work are essentially different and their BNF  will be convergent.
Furthermore,  we can choose the Birkhoff normal form to be either
\begin{align}
 \label{eqN} &\hat{N}(r):=\langle \hat{\omega}(r_d),r \rangle \quad \text{with} \quad \hat{\omega}(s):=(\omega_1+s,\omega_2,\ldots,\omega_d),\; \; \text{or} \\ 
\label{eqN2}
 &\bar N(r):=\langle \bar{\omega}(r_d),r \rangle \quad \text{with} \quad \bar{\omega}(s):=(\omega_1+s,\omega_2+{s^2},\ldots,\omega_{d-1}+s^{d-1},\omega_d). 
\end{align}
For sufficiently  Liouville $\omega$ we will have some constructions with $N_H=N$ where 
\begin{equation} \label{eqNdeg}
 {N}(r):=\langle {\omega},r \rangle. \end{equation}

\begin{Main} \label{theorem.unstable} 
For any $\omega \in \R^d$, $d\geq 3$, there exists a real entire Hamiltonian $H$  
as in \eqref{HH} such that the QP torus   $\cT_0$ is Lyapunov unstable. 

Moreover, the BNF of $H$ at $\cT_0$ can be chosen to be $\hat{N}(\cdot)$ or $\bar N(\cdot)$. 
In the latter case, if $\omega$ is Diophantine, then $\cT_0$ is KAM stable.  
\end{Main}  

While constructing these examples of Lyapunov unstable QP tori, we 
clarify several questions regarding the stability of QP motion in the analytic context. Namely, 

\begin{itemize}

 \item[$i)$] \bk Lyapunov instability of $\cT_0$ can be obtained for arbitrary frequencies $\omega \in \R^d$.

 \item[$ii)$] \bk The examples of Theorem \ref{theorem.unstable} have a convergent BNF, thus answering positively the question of Eliasson  mentioned in Section \ref{theorem.unstable} ($ii)$
(see \cite{E1,eliasson,EFK}). 
The same question in the case of elliptic fixed points is still open (see \cite{Fpoint,FKicm} for a discussion of this problem). 

\item[$iii)$] \bk The BNF can be chosen to be a very simple polynomial as in \eqref{eqN}.
This shows that  R\"{u}ssmann's local integrability result for Diophantine QP tori \cite{russmann}, 
that holds true when  the BNF is completely degenerate (equal to a function of $\langle \omega,r\rangle$), does not hold for a simple highly degenerate form as $\hat{N}$. 

\item[$iv)$] \bk The Birkhoff normal form $\bar N$ is non-degenerate in the sense of R\"ussmann. Hence, Theorem \ref{mB} proves in this case the  coexistence of diffusion and KAM stability.

\end{itemize}

\medskip

\begin{remark} Note that Herman conjectured that for Diophantine frequencies $\cT_0$ is accumulated by a positive measure of QP tori in the analytic category (see Section \ref{sec.KAM} below). If the conjecture is true then even the examples with BNF $\hat{N}$ should also have coexistence of Lyapunov instability and $\cT_0$ being accumulated by a positive measure of QP tori. 
\end{remark} 

 \subsection{Effective stability} 

An important question in classical mechanics is to estimate the escape rate of orbits starting in small neighborhoods of invariant objects such as fixed points or invariant tori. In our context we introduce, for a given $H$ as in \eqref{HH} and $\cT_0$,
\begin{equation}
\label{eq_def_TR}
 T(r):=\inf_{\theta\in \T^d, \abs{r'}\leq r}\left\{t>0 \; | \; \text{dist}(\Phi^t_{H}(\theta, r'), \cT_0)= \frac{1}{r}\right\}.
\end{equation}
If $T(r)$\footnote{We apologize for the double use of the notation $r$ as a scalar in definition \eqref{eq_def_TR} and previously as a variable in $\R^d$.} exists for all $r>0$ sufficiently small then we say that $\cT_0$ is {\it diffusive}.
Based on the Diophantine exponent $\tau$, exponential lower bounds for $T(r)$ can be derived from  estimates on the remainder terms in the BNF reductions. 
It follows from \cite{jorba,popov,wiggins} that for $H$ as in \eqref{HH}, $\omega \in \Omega^d_{\gamma,\tau}$, there exist
positive constants $C, R$ such that for $r<R$

 \begin{equation}\label{bound.DC}
  T(r)\geq r^{-1}exp (Cr^{-(\tau+1)^{-1}}).
 \end{equation}
One aim of this paper is to prove the optimality of the exponent in this bound for a certain class of Diophantine frequencies (see Corollary \ref{cor_optimal}).  Many results on optimal Arnold diffusion times  in smooth, Gevrey and analytic context exist in the literature, that sometimes relate the speed of diffusion to arithmetic conditions. We refer to  \cite{jeanpierre,LM05,Zh11,KLS14,FMS17} and references therein. 

\begin{remark}
 The usual definition of $T$ in \eqref{eq_def_TR} requires diffusion up to distance  $2r$ instead of $r^{-1}$, and this is the original context in which \eqref{bound.DC} was proved. In our case, the definition of diffusiveness with $r^{-1}$ instead of $2r$ does not change  the order of magnitude of the diffusion time and has the advantage of implying  Lyapunov instability of a diffusive torus $\cT_0$.
\end{remark}

Our main result on diffusion time is stated for vectors  
$\omega=(\tilde{\omega}, \omega_d)$, where $d\geq 3$ and $\tilde{\omega} \in \R^{d-1}$ does not belong to some Diophantine class.

\begin{Main}
\label{thm_diff_general}
 For any $\tau>0$, $C>0$ and any $\omega = (\tilde{\omega}, \omega_d)\in \R^d$, where $d\geq 3$ with $\tilde{\omega} \notin \Omega_{\tau}^{d-1}$,
there is a real analytic Hamiltonian $H$ as in
\eqref{HH} such that $\mathcal{T}_0$ is  diffusive and 
$T(r_n)\leq \exp (C r_n^{-(\tau+1)^{-1}})$ for a sequence $r_n\to 0$. Moreover, the BNF at $\cT_0$ is given by $\hat{N}(\cdot)$.
\end{Main}

The BNF of the Hamiltonians that we construct in 
Theorem \ref{thm_diff_general} must be very special. 
Indeed, it was proven in \cite{MG,BFN} that a QP  torus with Diophantine
frequency is generically and prevalently doubly exponentially stable. 
More precisely, it was shown  that a point that starts at distance $r$ from   the torus remains within distance $2r$ close to it
for an interval of time which is larger than $\exp(\exp(C {r}^{-(\tau+1)^{-1}})).$
The proof of double exponential stability is based on a combination of the estimates on the BNF and Nekhoroshev stability theory.  
It is worth mentioning that analogous results have been proved in the context of elliptic fixed points as well, both for exponential and double exponential stability (see  \cite{DF}, \cite{BFN2}).
To show how Theorem \ref{thm_diff_general} allows to approach the known lower bound on the diffusion speed  $T(r)$ for {\it some} Diophantine vectors we will need the following simple arithmetic lemma. 

\begin{lemma}
\label{lemma_diophantine}
 For any $\tau>d-1$ and $\tilde{\omega}\in \Omega^{d-1}_{\tau}$, $a.e$ $\omega_d\in \R$ 
 satisfies $\omega:=(\tilde{\omega}, \omega_d) \in \Omega^{d}_{\tau}$. 
\end{lemma}

Hence, if we pick $\tilde{\omega} \in \Omega^{d-1}_{\tau} \setminus \Omega^{d-1}_{\tau-\varepsilon}$, it is possible to ``extend'' it into $\omega=(\tilde{\omega}, \omega_d)$ for some $\omega_d$ such that ${\omega} \in \Omega^{d}_{\tau}$.

\begin{Coro}
\label{cor_optimal}
For any $\tau>d-1$, $\varepsilon>0$ there is a real analytic Hamiltonian $H$ as in
\eqref{HH} with $\omega \in \Omega^{d}_{\tau}$
such that $\mathcal{T}_0$ is  diffusive and $T(r_n)\leq \exp (C r_n^{-(\tau+1-\varepsilon)^{-1}})$ for a sequence $r_n\to 0$. 
\end{Coro}

Then due to Lemma \ref{lemma_diophantine} the proof of Corollary \ref{cor_optimal} becomes a direct application of Theorem \ref{thm_diff_general}.
The proof of  Lemma \ref{lemma_diophantine} is elementary.  We sketch it for completeness. 

\begin{proofof}{Lemma \ref{lemma_diophantine}}
  Let $I$ be an arbitrary bounded interval in $\R$.  
 We denote by $D_{\omega,\tau, \gamma}$ the set of  $\omega_d \in I$ satisfying 
 $(\tilde{\omega},\omega_d) \in \Omega^{d}_{\tau, \gamma}$. For any $k=(k_1, \ldots, k_d)\in \Z^{d}$, $k\neq 0$, consider the set
 
 \begin{equation*}
  A^{\omega}_{\tau, \gamma,k}=\left\{ \omega_d\in I \; | \; \abs{\langle k,\omega\rangle }<\frac{\gamma}{\norm{k}^{\tau}}\right\}.
 \end{equation*}
Since $\tilde{\omega}\in \Omega^{d-1}_{\tau}$, we have that 
$$I\setminus D_{\omega, \tau, \gamma}\subset \bigcup_{k \in \Z^d, k_d \neq 0}A^{\omega}_{\tau,\gamma, k}.$$
 Hence, for some constant $C_d>0$ we get that 
 \begin{align*}
 \mu(I\setminus D_{\omega, \tau, \gamma})&\leq \sum_{\tilde{k} \in \Z^{d-1}} \sum_{0<|k_d|< \norm{\tilde k}}\mu({A^{\omega}_{\tau,\gamma, k}})+  \sum_{|k_d|>0} \sum_{ \norm{\tk}\leq |k_d|} \mu({A^{\omega}_{\tau,\gamma, k}})\\
&\leq 2 \gamma \sum_{\tilde{k} \in \Z^{d-1}}\sum_{0<|k_d|< \norm{\tk}}  \cfrac{1}{|k_d|\norm{\tk}^{\tau}}+2\gamma      
  \sum_{|k_d|>0} \sum_{ \norm{\tk}\leq |k_d|}
 \cfrac{1}{|k_d|^{\tau+1}}\\
&\leq C_d \gamma \sum_{ \tilde{k} \in \Z^{d-1} \setminus \{0\}}  \cfrac{\ln \norm{\tk}}{\norm{\tk}^{\tau}}+ {C}_d\gamma\sum_{|k_d|>0} \cfrac{|k_d|^{d-1}}{|k_d|^{\tau+1}}=\mathcal O(\gamma).
 \end{align*}
 Therefore $\mu\left(I \setminus \bigcup_{\gamma>0}{D_{\tau,\gamma, \omega}}\right)=0$.
\end{proofof} \bk

\noindent {\bf \bg Liouville frequencies.} For elliptic fixed points with non resonant frequencies of smooth Hamiltonians,
the existence of the BNF up to arbitrary order implies that the diffusion time from small $r$-neighborhoods of the origin cannot be faster than arbitrarily high powers in $r^{-1}$. For sufficiently Liouville frequencies, finite order BNFs may be not well defined at an invariant torus, even for real analytic Hamiltonians. In this case, diffusion time may be much faster than in the case of elliptic equilibria. We will work with non resonant frequencies $\omega=(\tilde{\omega}, \omega^d) \in \R^d$, $d\geq 3$ where  $\tilde{\omega}\in \R^{d-1}$ is such that 
there is a sequence $\{\bar k_j\} \subset \Z^{d-1}$ satisfying
\begin{equation}
\label{eq_liouville_cond}
\lim_{j\to \infty} \frac{\ln \abs{\langle \tilde{\omega}, \bar k_j \rangle}}{\norm{\bar k_j}}=-\infty.
\end{equation}

 \begin{Main}
 \label{thm_fast}
 For any $\omega \in \R^d$ satisfying \eqref{eq_liouville_cond}:

\medskip 

{ a)}  There exists a real entire Hamiltonian $H$ as in $\eqref{HH}$ with the BNF of $H$ at $\cT_0$ given by $N(\cdot)=\langle \omega,\cdot\rangle$, and such that  $\mathcal{T}_0$ is  diffusive with 
$T(r)\leq {r_n^{-n}}$ for a sequence $r_n \to 0$. 

\medskip 
 
{ b)}  There exists a real entire Hamiltonian $H$ as in $\eqref{HH}$ such that $\mathcal{T}_0$ is  diffusive and 
$T(r_n)\leq {r_n^{-4}}$ for a sequence $r_n \to 0$.

\begin{remark}
 Notice that the difference between the two results in Theorem \ref{thm_fast} is that having a faster diffusion in $b)$ comes with the price of not having a well defined BNF as we do have in $a)$. When, in Theorem \ref{th03}, we will state the explicit constructions for both $a)$ and $b)$, we will explain in Remark \ref{remark.explain.bnf} the reason behind the slowing down of the diffusion in $a)$.
\end{remark}

\begin{remark} It is easy to see from our proof that if we just ask to diffuse from an initial condition $\norm{z_n}=r_n$ to $n$ and not $r_n^{-1}$, 
then it is possible to replace the upper bound ${r_n^{-4}}$ of case ${ b)}$ by ${r_n^{-2-\epsilon}}$, with $\epsilon>0$ arbitrarily small. Moreover,
if we assume stronger Liouville conditions on $\tilde{\omega}$ we can even get diffusion times that are even closer to  ${r_n^{-2}}$, which is clearly a lower bound for diffusion times for $H$ as in \eqref{HH}.
\end{remark} 

\end{Main}

\subsection{Coexistence of diffusion and integrability} 

A natural question in Hamiltonian dynamics is whether a real analytic Hamiltonian system can be integrable on 
an open set of the phase space and not completely integrable. 

One aim of this paper is to show that such examples do exist. We actually construct real analytic Hamiltonians 
that are analytically integrable on half of the phase space while all orbits on the other side accumulate at infinity. 
We will work with non resonant frequencies  satisfying \eqref{eq_liouville_cond}. The main result is the following.

 \begin{Main}
 \label{thm_teorema_estrany}
 For any $\omega \in \R^d$ satisfying \eqref{eq_liouville_cond} there exists a real entire Hamiltonian $H$ as in $\eqref{HH}$ such that:

 \begin{itemize}
 \item[$i)$] \bk There exists a real analytic symplectic diffeomorphism $\Psi$ from $M^-=\T^d \times \R^{d-1}\times (-\infty,0)$ to itself, such that on $M^-$ we have 
$H\circ \Psi=H_{0}:=\langle \omega, r\rangle$.
 \item[$ii)$] \bk For any $(\theta, r)\in \T^d \times \R^{d-1}\times (0,\infty)$, we have $\lim \sup_{t\rightarrow \infty}{\abs{\Phi_H^{t}(\theta, r)}}=\infty$.
 \end{itemize}
 The BNF of $H$ at $\cT_0$ is given by $N(\cdot)=\langle \omega,\cdot\rangle$.
\end{Main}

\bigskip 

The question of coexistence of integrability and diffusion for analytic systems 
remains completely open if integrability is required to be non-degenerate (twist integrability for example). With a similar construction to that of Theorem \ref{thm_teorema_estrany}, we can obtain the following examples. 

\vspace{0.2cm}

\begin{Main}
 \label{thm_teorema_Cr}
 For any $\omega \in \R^d$ satisfying \eqref{eq_liouville_cond}, for any $l \in \N^*$, there exists a real entire Hamiltonian
 $H$ as in $\eqref{HH}$ and a symplectic diffeomorphism $\Psi$ on $\T^d \times \R^{d}$, that is of class $C^l$ but not of class $C^{l+1}$, such that $H\circ \Psi=H_{0}:=\langle \omega, r\rangle$. The BNF of $H$ at $\cT_0$ is given by $N(\cdot)=\langle \omega,\cdot\rangle$.
\end{Main}

Observe that the main ingredient in the proof of Theorem \ref{thm_teorema_Cr} (see below the statement and proof of Theorem \ref{th3bis} that gives the explicit construction for Theorem \ref{thm_teorema_Cr}) is a fine tuning of the effect of the almost resonances of $\omega$ on the instabilities of a Hamiltonian as in \ref{HaHaHaHa}. This fine tuning has the effect of maintaining linearizability in class $C^l$ but destroying it in class $C^{l+1}$. An analogy can be seen with Sternberg's linearization theorem near a hyperbolic fixed point that gives $C^l$ regularity of the linearization provided a sufficient number, related to $l$, of non-resonance conditions hold \cite{sternberg}.   

 \subsection{KAM stability} \label{sec.KAM}

It was conjectured by Herman (see \cite{herman_icm}) that, without any non-degeneracy condition, 
a Diophantine KAM torus of an analytic Hamiltonian  is accumulated by a set of positive measure of KAM tori. Herman's conjecture is known to be true in two degrees of freedom \cite{russmann}, but remains open in general,
with some progress being made in \cite{EFK}, where it is shown that an analytic invariant torus $\cT_0$ with 
Diophantine frequency $\om$ is never isolated from other KAM tori.

Herman's conjecture on KAM stability of a Diophantine equilibrium or QP torus  is  known to be true in the smooth category for $d=2$ due to Herman's last geometric theorem (see \cite{FKrik}). 
Counter-examples to the conjecture in $C^\infty$ and with arbitrary frequencies were build in \cite{EFK} for $d\geq 4$, and later in \cite{FSpoint} for $d=3$. 

One aim of this work is to build, starting from $3$ degrees of freedom and for sufficiently Liouville frequencies $\omega$,
real analytic  Hamiltonians that have QP tori with frequency $\omega$ that are not accumulated by a set of positive measure of KAM tori. This shows that some arithmetic condition in Herman's conjecture is indeed necessary.

\begin{Main}
 \label{thm_isolated}
For any $\omega \in \R^d$, $d\geq 3$ satisfying \eqref{eq_liouville_cond} there exists a real entire Hamiltonian $H$ as in \eqref{HH} such that for any $(\theta, r)\in \T^d \times \R^d$ with $r_d \neq 0$, 
 $$\lim \sup_{t\rightarrow \infty}{\abs{\Phi_H^{t}(\theta, r)}}=\infty.$$ 
The BNF of $H$ at $\cT_0$ is given by $N(\cdot)=\langle \omega,\cdot\rangle$.
\end{Main}

Note that Bounemoura proved in \cite{Bounemoura} that an invariant quasi-periodic torus is KAM-stable under the hypothesis that the Hamiltonian is sufficiently smooth and has a non-degenerate  Hessian matrix of its BNF of degree 2 (that part of the BNF is defined for all non-resonant frequencies). In our example, the entire BNF can be defined and is in fact equal to $\langle \omega, r\rangle$. 
Theorem \ref{thm_isolated} thus shows that R\"ussmann's local integrability result of Diophantine QP tori with a degenerate BNF cannot be generalized to the case of sufficiently Liouville vectors. 

\medskip 

\begin{remark} In our construction $\cT_0$ is not isolated, the  hyperplane $r_d=0$ is foliated 
 by invariant tori with frequency $\om$. In \cite{EFK} it was proved that Diophantine analytic QP tori are always accumulated by other QP tori. 
 The question of the existence of Liouville QP tori that are completely isolated is still open, even for smooth Hamiltonians. 
\end{remark}

\section{Constructions} 
\label{sec_constructions}

Given $\om \in \R^d$, all our examples will have the form:
\begin{equation}
\label{eq_ham2} 
 H=\lim_{n\rightarrow \infty}{H_n}, \quad H_n(\theta,r)=\langle  \om(s),r \rangle -\sum_{j=2}^{n}\phi_j(s) \sin (2\pi \langle k_j, \tilde{\theta} \rangle).
\end{equation}
We can now give in Theorems \ref{th1}--\ref{th4}   the specific forms of the Hamiltonians that will satisfy Theorems \ref{theorem.unstable}--\ref{thm_isolated}.
Theorem \ref{theorem.unstable} can be rewritten as follows.

\begin{theorem} \label{th1} Let $\omega \in \R^d$, $d\geq 3$. Choosing $\omega(\cdot)$ to be $\hat \omega(\cdot)$ (or $\bar \omega(\cdot)$),  there exists a sequence $\{k_j\} \subset \Z^{d-1}$, such that the Hamiltonian in \eqref{eq_ham2}
with $\phi_j(s)={s^j} e^{-j \norm{k_j}}$ satisfies the first (or second) conclusion of Theorem \ref{theorem.unstable}. 
\end{theorem}

Although Theorem \ref{th1} holds for all frequencies its proof depends on whether the frequency is resonant or not and also on the form of $\omega(\cdot)$.
Different sequences must be constructed in the proof for the different cases.

Consider next $\omega = (\tilde{\omega}, \omega^d)$ with $\tilde{\omega} \notin \Omega^{d-1}_{\tau}$, $d\geq 3$. 
Then up to a permutation of indices for $\omega$ Theorem \ref{thm_diff_general} can 
be without loss of generality restated as follows.

\begin{theorem} \label{th2} For any $C, \tau>0$, there exists a sequence 
$\{k_j\}\subset \Z^{d-1}$ such that for $\phi_j(s)={s^j} e^{-\frac{C}{2} \norm{k_j}}$ and $\omega(\cdot)=\hat{\omega}(\cdot)$ the Hamiltonian in \eqref{eq_ham2} belongs to $C^{\omega}_{\rho}$, $\rho=\frac{C}{8\pi d}$, and satisfies the conclusion of Theorem \ref{thm_diff_general}.
\end{theorem}

We pass now to the purely Liouville constructions of Theorems   \ref{thm_fast}--\ref{thm_isolated}.

\begin{theorem} \label{th03} For any $\omega \in \R^d$, $d\geq 3$ satisfying \eqref{eq_liouville_cond}, there exists a sequence $\{k_j\}\subset \Z^{d-1}$   such that:

\noindent { a)} If $\phi_j(s)={s^j} e^{-j \norm{k_j}}$ and $\omega(\cdot)\equiv \omega$
then the Hamiltonian in \eqref{eq_ham2} satisfies {a)} of Theorem \ref{thm_fast}.

\medskip 

\noindent { b)}  If  $\phi_j(s)={s^2} e^{-j \norm{k_j}}$ and $\omega(\cdot)\equiv \omega$
then the Hamiltonian in \eqref{eq_ham2} satisfies {b)}  of Theorem \ref{thm_fast}.

\end{theorem} 

 \begin{remark} \label{remark.explain.bnf} We will see in Section \ref{sec.bnf} why taking powers $s^j$ in $\phi_j(s)$ as in $a)$ is required to guarantee that the BNF of $H$ at $\cT_0$ is given by $N(\cdot)=\langle \omega,\cdot\rangle$. Of course, this has the inconvenient of slowing down the diffusion compared to the definition of $\phi_j(s)$ with an $s^2$ as in $b)$. \end{remark}

\begin{theorem} \label{th3} For any $\omega \in \R^d$, $d\geq 3$ satisfying \eqref{eq_liouville_cond}, there exists a sequence $\{k_j\}\subset \Z^{d-1}$   such that if 
$\phi_{j}(s)= \langle \tilde{\om}, k_j \rangle {s^j} e^{\norm{k_j}s}$ and $\omega(\cdot)\equiv \omega$
then the Hamiltonian in \eqref{eq_ham2} satisfies the conclusion of Theorem \ref{thm_teorema_estrany}.
\end{theorem} 

We also have

\begin{theorem} \label{th3bis} For any $\omega \in \R^d$, $d\geq 3$ satisfying \eqref{eq_liouville_cond}, there exists a sequence $\{k_j\}\subset \Z^{d-1}$ such that if $\phi_{j}(s)= \langle \tilde{\om}, k_j \rangle {s^j} \norm{k_j}^{- l -1} j^{-2}$ and $\omega(\cdot)\equiv \omega$
then the Hamiltonian in \eqref{eq_ham2} satisfies the conclusion of Theorem \ref{thm_teorema_Cr}.
\end{theorem}

A simple modification of the construction in Theorem \ref{th3} gives a real entire Hamiltonian with a QP 
torus of Liouville frequency that is not accumulated by a positive measure set of KAM tori. 

\begin{theorem} \label{th4} For any $\omega \in \R^d$, $d\geq 3$ satisfying \eqref{eq_liouville_cond}, there exists a sequence $\{k_j\}\subset \Z^{d-1}$   such that if $\phi_{j}(s)= \langle \tilde{\om}, k_j \rangle {s^j} e^{\norm{k_j}{s^2}}$ and $\omega(\cdot)\equiv \omega$
then the Hamiltonian in \eqref{eq_ham2} satisfies the conclusion of Theorem \ref{thm_isolated}.
\end{theorem} 

\section{Proofs} 

For convenience of the presentation we summarize the choices made in the various constructions of Theorems \ref{th1}--\ref{th4}. Recall that $H_n$ are constructed as in \eqref{eq_ham2}, with 
$\{k_j\}\subset \Z^{d-1}$ a strictly increasing sequence and  the following possibilities for $\phi_j$ :
\begin{itemize}
 \item[$i)$] \bk $\om(\cdot) : \R \to \R^d$ is either $\hat{\om}$ or $\bar{\om}$ and $\phi_j(s)=s^j e^{-j\norm{k_j}}$,
 \item[$ii)$] \bk  $\om(\cdot) : \R \to \R^d$ is $\hat{\om}$ and $\phi_j(s)=s^j e^{-\frac{C}{2}\norm{k_j}}$, for some $C>0$,
 \item[$iii)$] \bk  $\om(\cdot) \equiv  \omega$ and $\phi_j(s)=s^j e^{-j\norm{k_j}}$,
 \item[$iv)$] \bk $\om(\cdot) \equiv  \omega$ and $\phi_j(s)=s^2 e^{-j\norm{k_j}}$,
 \item[$v)$] \bk  $\om(\cdot) \equiv  \omega$ and $\phi_j(s)=\langle \tilde{\omega}, k_j\rangle s^j e^{\norm{k_j}s}$,
 \item[$vi)$] \bk $\om(\cdot) \equiv \omega$ and $\phi_j(s)=\langle \tilde{\omega}, k_j\rangle s^j \norm{k_j}^{-l -1} j^{-2}$, 
 \item[$vii)$] \bk $\om(\cdot) \equiv \omega$ and $\phi_j(s)=\langle \tilde{\omega}, k_j\rangle s^j e^{\norm{k_j}s^2}$.
\end{itemize}

\medskip

Let us now explain how the sequences $\{k_j\}$ will be chosen in the different cases. 

For cases $iii)$--$vii)$,  $\{k_j\}$ will be a fast growing subsequence of the sequence $\{\bar{k}_j\}$ satisfying \eqref{eq_liouville_cond}.  
For cases $i)$ and $ii)$ we will use the following elementary fact. 
\medskip

\begin{lemma}
\label{lemma_sequences}
 For any $\omega \in \R^d, \; d\geq 3$, assume $\omega(\cdot)$ satisfies either \eqref{eqN} or \eqref{eqN2}. 
 There exists a sequence $\{s_j\}\subset \R$ and an increasing sequence in norm $\{k_j\}\subset \Z^{d-1}$ such that
 \begin{itemize}
   \item[$a)$] \bk $\lim_{j\to \infty}{\abs{s_j}}=0$,
   \item[$b)$] \bk $\lim_{j\to \infty}{\norm{k_j}}=\infty$,
   \item[{$c)$}] \bk $\langle \tilde{\omega}(s_j) , k_j \rangle=0$.
 \end{itemize}
In case  $\tilde{\om} \notin \Omega^{d-1}_{\tau}$ and if $\omega(\cdot)$ satisfies \eqref{eqN}, we can assume without loss of generality that 
\begin{equation}\label{eq_dir2}  \norm{k_j}<\abs{s_j}^{-(\tau+1)^{-1}}.\end{equation}
\end{lemma}

\begin{proof}
Let us denote $\omega':=(\omega_1,\omega_2)$ (we only consider the two first components of $\omega$). We will divide the proof according to whether $\omega'$ is resonant or
 non-resonant. We will only treat the case where $\omega(\cdot)$ is as in \eqref{eqN2}, the case \eqref{eqN} being similar albeit easier.
 \begin{itemize}
   \item[a)] \bk Assume first that $\omega$ is such that $\omega'$ is non-resonant, $\om(\cdot)$ as in \eqref{eqN2}. 
 By Dirichlet's Theorem there exists $C>0$ and an increasing sequence in norm $\{k_i'\}\subset \Z^{2}$, 
   $k_i'=(k_{i,1}, k_{i,2})$ such that
  \begin{equation}
  \label{eq_ineq_sec1}
   \abs{\langle \om', k_i'\rangle}<\frac{C}{\norm{k_i'}}.
  \end{equation}
Consider  $k_i:=(k_{i,1}, k_{i,2}, 0, \ldots, 0)$. Now  $\langle \tilde{\omega}(s_i), k_i \rangle=0$ is equivalent to 
  \begin{equation}
    \label{eq_sec_deg1}
    k_{i,2}s_i^2 + k_{i,1}s_i + \langle {\omega'}, k'_i \rangle=0,
  \end{equation}
which is easily seen to have a solution $s_i \to 0$ as required. 
 
 \bigskip
 
  \item[b)] \bk Assume now that $\omega$ is such that $\omega'$ is resonant, $\om(\cdot)$ as in \eqref{eqN2}.
  There exists $m=(m_1,m_2)$ such that $\langle m , \omega' \rangle=0$.
  Then for an increasing sequence $\{a_i\}\subset \N$ we define
  \begin{equation*}
    k_i:=(a_i m_1+1, a_i m_2, 0, \ldots, 0)\in \Z^{d-1}.
  \end{equation*}
  The equation $\langle \tilde{\omega}(s_i), k_i\rangle=0$ is  then equivalent to 
  \begin{equation*}
     s_i k_{i,1}+s_i^2 k_{i,2}=-\langle \tilde{\omega},k_i \rangle=-\omega_1,
  \end{equation*}
which clearly has a solution $s_i \to 0$ as required. 
\end{itemize}

In the non-resonant case, and  $\om(\cdot)$ as in \eqref{eqN}, equation \eqref{eq_sec_deg1} becomes $k_{i,1}s_i + \langle {\omega'}, k'_i \rangle=0$, solved by $s_i :=-\langle \tilde{\omega}, k_i\rangle /k_{i,1}$. If $\tilde \omega \notin \Omega_{\tau}^{d-1}$, we can assume without loss of generality that $ \abs{\langle \tilde{\omega}, k_i\rangle } < {\norm{k_i}^{-\tau}}$ and $\abs{k_{i,1}}=\norm{k_i}$. Thus \eqref{eq_dir2} holds.
 
\end{proof} \bk

\subsection{Convergence}  The following settles the convergence question in Theorems \ref{th1}--\ref{th4}.  
\begin{proposition}  \label{prop.convergence} In cases $i)$, $iii)$--$vii)$ the convergence $H_n \to H$ holds in the $C^\omega_{\rho}$ topology for any $\rho>0$, hence the limit $H$ is real entire. In case $ii)$, the convergence holds in $C^\omega_{\bar \rho}$ for $\bar \rho=\frac{C}{8\pi d}$, hence the limit $H \in C^\omega_{\bar \rho}$.
\end{proposition} 

\begin{proof}
\noindent {\it Cases $i)$--$iv)$.}  We treat the case $i)$, the other cases being similar. According to \eqref{eq_ham2}, we have that  for any $\Delta, \rho>0$, there exists $N\in \N$ such that for all $m>n \geq N$
\begin{align*}
\norm{H_m-H_n}_{\Delta,\rho}&\leq
\sum_{j=n+1}^{m}(\Delta+\rho)^je^{-j\norm{k_j}}{\norm{\sin(2\pi \langle k_j, \tilde{\theta} \rangle)}_{\rho}}\\
&<  \sum_{j=N}^{\infty}{(\Delta+\rho)^j e^{-\norm{k_j}(j-2\pi d \rho)}}.
\end{align*}
Therefore $\{ H_n\}$ is a Cauchy sequence in $C^{\omega}_{\Delta,\rho}$. Since $\Delta,\rho>0$ are arbitrary, 
the limit $H$ is a real entire function. 
\medskip 

\noindent {\it Cases $v)$--$vii)$.}  We treat case $v)$, the other cases being similar. From condition \eqref{eq_liouville_cond}, there exists a sequence $u_j \to \infty$ such that 
$$\ln \abs{\langle \tilde{\omega}, k_j \rangle} \leq -u_j{\norm{k_j}}.$$
For any $\Delta, \rho>0$, for all $ \; \varepsilon >0$ there exists 
$N\in \N$ such that for all $m>n\geq N$ 
 \begin{equation*}
 \label{eq_real_entire}
 \norm{H_m - H_n}_{\Delta ,\rho}
 \leq  \sum_{j=n+1}^{m}{(\Delta+\rho)^je^{\norm{k_j}(\Delta+(2\pi d+1)\rho-u_j)}}<\varepsilon.
 \end{equation*}
 Therefore $\{ H_n\}$ is a Cauchy sequence in $C^{\omega}_{\Delta,\rho}$. Since $\Delta,\rho>0$ are arbitrary, 
the limit $H$ is a real entire function. 
\end{proof} \bk

 \subsection{Birkhoff normal forms} \label{sec.bnf}

\begin{proposition} \label{prop.BNF} In Theorems \ref{th1}--\ref{th4}, and except for Theorem \ref{th03} b), the BNF at $\cT_0$ is defined and equals $\langle \omega(r_d), r \rangle$. 
\end{proposition}

\begin{proof}
Define  $\Psi_n$ to be the canonical transformations obtained via the generating functions
\begin{equation}
\label{eq_generating_function}
 S_{n}(\Theta, r)= \langle \Theta, r \rangle- 
 \frac{1}{2\pi}\sum_{j=2}^n{ \langle \tilde{\omega}(s), k_j \rangle^{-1}\phi_j(s)  \cos(2\pi \langle k_j, \tilde{\Theta} \rangle)}, \;
\end{equation}
which is a real analytic function near the origin. More explicitly for all $n\in \N$ we obtain the change of variables $(\Theta_n, R_n)=\Psi_{n}(\theta, r)$ given by the equations
\begin{equation*}
 \begin{array}{l}
  \tilde{R}_n=\cfrac{\partial S_n(\Theta_n, r)}{\partial \tilde{\Theta}_n}=\tilde{r}+
  \sum_{j=2}^{n}{k_j \langle \tilde{\omega}(s), k_j \rangle^{-1} \phi_j(s) \sin(2\pi \langle k_j, \tilde{\Theta}_n \rangle)},\\
  \\
  R_{d,n}=\cfrac{\partial S_n(\Theta_n, r)}{\partial \Theta_{d,n}}=s,\\
  \\
  \tilde{\theta}= \cfrac{\partial S_n(\Theta_n, r)}{\partial \tilde{r}}= \tilde{\Theta}_n,\\
  \\
  \theta_d=\cfrac{\partial S_n(\Theta_n, r)}{\partial s}=\Theta_{d,n}-
  \cfrac{1}{2\pi}\sum_{j=2}^{n} \partial_s\left(\langle \tilde{\omega}(s), k_j \rangle^{-1}\phi_j(s)\right)\cos(2\pi \langle k_j, \tilde{\Theta}_n \rangle).
 \end{array}
\end{equation*} 
Thus
$$H_n= H_0 \circ \Psi_n,$$
where $H_0= \langle \omega(r_d), r \rangle$.
In fact, we can define in a formal way 
 \begin{equation*}
\label{eq_generating_function_infty}
 S_{\infty}(\Theta, r)= \langle \Theta, r \rangle- 
 \frac{1}{2\pi}\sum_{j=2}^\infty{ \langle \tilde{\omega}(s), k_j \rangle^{-1}\phi_j(s)  \cos(2\pi \langle k_j, \tilde{\Theta} \rangle)}, \;
\end{equation*}
which formally conjugates the limit Hamiltonian $H$ to $H_0$. We only need to verify that 
$$f=  \frac{1}{2\pi} \sum_{j=2}^\infty{ \langle \tilde{\omega}(s), k_j \rangle^{-1}\phi_j(s)  \cos(2\pi \langle k_j, \tilde{\Theta} \rangle)} \in \mathcal C^{\omega}(\T^d_\rho)[[r]]\cap \cO(r^2).$$
When $\phi_j(s)=c_j{s^j}$ as in Theorems \ref{th1}, \ref{th2}, and $\omega(\cdot)=\hat{\omega}(\cdot)$, the coefficient of ${s^p}$ in the power series of $f$ is the trigonometric polynomial\footnote{ This is exactly where the increasing powers $s^j$ in the definition of $\phi_j$ play a decisive role in controlling the Birkhoff Normal Form at $\cT_0$.} given by 
$$ \frac{1}{2\pi} \sum_{\overset{l\geq 0, j\geq 2}{l+j=p}}\frac{c_j}{\langle k_j,\tilde{\omega}\rangle} \left(-\frac{k_{j,1}}{\langle k_j,\tilde{\omega}\rangle}\right)^{l}\cos(2\pi \langle k_j, \tilde{\Theta} \rangle).$$

In the other situations, for example  when $\omega(\cdot)\equiv\omega$ and $\phi_j(s)=\langle \tilde{\om}, k_j \rangle  {s^j}e^{\norm{k_j}s}$ as in Theorem \ref{th3}, 
then the coefficient of ${s^p}$ in the formal power series of $S_{\infty}(\Theta, r)$ is the trigonometric polynomial given by 
$$ \frac{1}{2\pi}\sum_{\overset{l\geq 0, j\geq 2}{l+j=p}} \norm{k_j}^l\cos(2\pi \langle k_{j}, \tilde{\Theta} \rangle).$$
The other cases are similar. \end{proof} \bk 
We now consider the cases of Theorems \ref{th3} and \ref{th3bis} where the conjugacies to the degenerate BNF $\langle \omega,r \rangle$ do converge.

\begin{proposition}  $ \ $ 
\label{prop.conjugacy}
 In case { $v)$}, the map $\Psi=\lim \Psi_{n}$ with $\Psi_ n$ as in \eqref{eq_generating_function}
 is well defined on $M^-=\T^d \times \R^{d-1}\times \R^{-}$ and is a real analytic symplectic diffeomorphism from $M^-$ to itself. 

In case {$vi)$}, the map $\Psi=\lim \Psi_{n}$ with $\Psi_ n$ as in \eqref{eq_generating_function}
 is well defined on $M=\T^d \times \R^{d}$ and is a  diffeomorphism from $M$ to itself that is of class $C^{l}$ but not of class $C^{l+1}$. 
\end{proposition}

\begin{proof}  We start with case {$v)$}. From the definition \eqref{eq_generating_function} we have that $\Psi_n$ is generated by 
\begin{equation*}
 S_{n}(\Theta, r)= \langle \Theta, r \rangle- 
 \frac{1}{2\pi}\sum_{2\leq j\leq n}   s^j e^{\norm{k_j}s}  \cos(2\pi \langle k_j, \tilde{\Theta} \rangle), \;
\end{equation*}
that preserves for every $\rho>0$ the domain $M^-_\rho=\T^d \times \R^{d-1}\times (-\infty,-\rho)$. Moreover, $S_n$ converges in $C^\omega_{\frac{\rho}{10d}}$ on  $M^-_\rho$.
Hence $\Psi_n$ defines a real analytic symplectic diffeomorphism on every  $M^-_\rho$, $\rho>0$ (we assume $k_2$ is sufficiently large and $\{k_j\}$ is fast growing). 

We treat now case { $vi)$}. In this case $\Psi_n$ is generated by 
\begin{equation*}
 S_{n}(\Theta, r)= \langle \Theta, r \rangle- 
 \frac{1}{2\pi}\sum_{2\leq j\leq n}   s^j {\norm{k_j}}^{-l-1} j^{-2} \cos(2\pi \langle k_j, \tilde{\Theta} \rangle), \;
\end{equation*}
and it is clear that the limit  $\Psi=\lim \Psi_{n}$ is a diffeomorphism of $M$ of class $C^{l}$ but not of class $C^{{l}+1}$. 
\end{proof} \bk

\begin{remark} In principle, it should be possible to use our constructions to obtain examples that are Lyapunov stable but not KAM stable. A possible approach would be to replace the choice of $\phi_j$ in  { $vi)$} by $\phi_j(s)=\langle \tilde{\omega}, k_j\rangle s^j b_j$, with $|b_j|\leq 1$ chosen such that the resulting $\Psi_{n}$ forms a sequence of diffeomorphisms of $M$ such that 
$|\pi_2(\Psi_{n}(\theta,r))|\leq 10|r|$ for all $n$ while  $\Psi_{n}$ diverges in the $C^0$ topology in a way that guarantees the absence of  invariant tori besides the ones at $s=0$.  
\end{remark}

 \subsection{Fast approximations}
Let us denote by $\Phi_n^t(\cdot)$ the flow of $H_n$. It is clear that by choosing $\{k_n\}$ to grow sufficiently fast, one can guarantee that the flow of $H$ will be very close to the flow of $H_n$ during very long times. Thus, it is convenient to give finite time versions of all the properties required in Theorems \ref{th1}--\ref{th4} that we start by checking for the flow  $\Phi_n^t(\cdot)$. For fixed $C,\tau>0$, let us define the following conditions:

\begin{description}
\item[${(\mathcal{P}^1_{n})}$] There exists $z\in \T^d \times \R^d$ with $\norm{z}\leq \frac{1}{n}$ and $t>0$ s.t $\norm{\Phi^{t}(z)}>n$.

\item[${(\mathcal{P}^2_{n})}$]  There exists $z\in \T^d \times \R^d$ with $\norm{z}\leq \frac{1}{n}$ and $t \leq \exp(C {\norm{z}}^{-(\tau+1)^{-1}})$ satisfying $\norm{\Phi^{t}(z)}>{\norm{z}}^{-1}$.

\item[${(\mathcal{P}^3_{n})}$]  There exists $z\in \T^d \times \R^d$ with $r_n:=\norm{z}\leq \frac{1}{n}$ and $t \leq {r_n^{-2n}}$ satisfying $\norm{\Phi^{t}(z)}>{\norm{z}}^{-1}$.

\item[${(\mathcal{P}^4_{n})}$]  There exists $z\in \T^d \times \R^d$ with $r_n:=\norm{z}\leq \frac{1}{n}$ and $t \leq {r_n^{-4}}$ satisfying $\norm{\Phi^{t}(z)}>{\norm{z}}^{-1}$.

 \item[${(\mathcal{P}^5_{n})}$] For all $z \in Q^+_n$ there exists $t>0$ s.t  $\norm{\Phi^{t}(z)}>n$, where $Q^+_n:= \T^d \times [-n,n]^{d-1} \times \left[n^{-1},n \right]$.

  \item[${(\mathcal{P}^6_{n})}$] For all $z \in Q_n$ there exists $t>0$ s.t  $\norm{\Phi^{t}(z)}>n$, where $Q_n:= \T^d \times [-n,n]^{d-1} \times (\left[-n,-n^{-1} \right] \cup \left[n^{-1},n \right])$.

\end{description}

We will write the previous conditions with $n=\infty$ to indicate that they hold for all $n$ large enough.
All the proofs of Theorems \ref{th1}--\ref{th4} rely on the following Lemma.

\begin{proposition}
\label{prop.approximation}  For any $i\in \{1,2,3,4,5,6\}$,  if $k_2,\ldots,k_n$ are chosen and ${(\mathcal{P}^i_{n})}$ is satisfied by the flow of $H_n$, then if $k_{n+1}$ is chosen sufficiently large the flow of $H$ also satisfies ${(\mathcal{P}^i_n)}$.
\end{proposition}

\begin{proof} It follows from the Gronwall inequalities that the conditions ${(\mathcal{P}^i_{n})}$
are open in the $C^3$ topology on the Hamiltonian. Hence, the lemma follows from the fact that $\norm{H-H_n}_{C^3} \to 0$ as $k_{n+1} \to \infty$. 
\end{proof} \bk

\subsection{Diffusion at finite scales}
\label{sec_proof1}

We now verify the diffusion properties ${(\mathcal{P}^i_{n})}$ for the flows 
$\Phi_n^t(\cdot)$ of $H_n$ in the various cases. 

\medskip 

\begin{proposition} \label{prop.diffusion} There exists a sequence $\{k_n\}\subset \Z^{d-1}$ and $N>0$ such that for all $n\geq N$, the following holds:

In case { $i)$}, $\Phi_n$ satisfies $(\mathcal{P}_{n}^1)$. 

In case { $ii)$}, $\Phi_n$ satisfies ${(\mathcal{P}^2_n)}$. 

In case { $iii)$}, $\Phi_n$ satisfies ${(\mathcal{P}^3_n)}$. 

In case { $iv)$}, $\Phi_n$ satisfies ${(\mathcal{P}^4_n)}$. 

In case { $v)$}, $\Phi_n$ satisfies ${(\mathcal{P}^5_{n})}$.  

In case { $vii)$}, $\Phi_n$ satisfies ${(\mathcal{P}^6_n)}$.  

\end{proposition}

\begin{proof}
 We start with case $i)$. Consider the initial condition $z=( \theta,  r)$ with
 \begin{equation}
\label{eq_initial_condition}
   \theta=(0,\ldots, 0,0), \quad r=(0, \ldots, 0, s_n),
\end{equation}
where $\{s_n\}$ is the corresponding sequence for $\{k_n\}$ in Lemma \ref{lemma_sequences}. We can assume $\abs{s_n}\leq n^{-1}$, which implies $\norm{z}\leq n^{-1}$. 
It follows from the expression of the Hamiltonian $H_n$ that along the orbit of $z$ we have 
$\dot{s}=-\frac{\partial H_n}{\partial \Theta_d}=0$ and also $\dot{\tilde{\theta}}=\tilde{\omega}_n:=\tilde{\omega}(s_n)$, 
hence from { $c)$} in Lemma \ref{lemma_sequences} we have $\langle k_n, \tilde{\theta}(t)\rangle\equiv0$. Therefore the corresponding flow becomes 
$\Phi_{n}^t(z)=(\tilde{\omega}_n t, \theta_d(t), \tilde{r}(t), s_n)$ with
\begin{equation*}
\label{eq_r_flow}
 \tilde{r}(t)=A_n(t)+ B_n(t),
\end{equation*}
where
\begin{equation*}
\label{eq_A_B}
 \begin{array}{l}
   A_n(t)=2\pi k_n s_n^n e^{-n\norm{k_n}}  t,\\ 
   B_n(t)=\sum_{2\leq j< n}{s_n^j \cfrac{ e^{-j\norm{k_j}}  }{\langle k_j, \tilde{\omega}_n \rangle} k_j \sin(2\pi \langle k_j, \tilde{\omega}_n \rangle t)}.
 \end{array}
\end{equation*}
Then since $B_n(t)$ is bounded there exists $t>0$ such that $\norm{\tilde r(t)}>n$, which implies that $\Phi_n$ satisfies $( \mathcal{P}_{n}^1)$.

Consider now the case {$ii)$}, where $\tilde{\om} \notin \Omega^{d-1}_{\tau}$. For $z$ as in \eqref{eq_initial_condition}, $\Phi^t_n$ is as above with
\begin{equation*}
\label{eq_r_flow2}
 \tilde{r}(t)=A_n(t)+  B_n(t),
\end{equation*}
where now
\begin{equation*}
\label{eq_A_B2}
 \begin{array}{l}
    A_n(t)=2\pi k_n s_n^n e^{-C/2\norm{k_n}}  t,\\ 
    B_n(t)=\sum_{ 2\leq j< n}{s_n^j \cfrac{ e^{-C/2\norm{k_j}}  }{\langle k_j, \tilde{\omega}_n \rangle} k_j \sin(2\pi \langle k_j, \tilde{\omega}_n \rangle t)}.
 \end{array} 
\end{equation*}
Then if $t:=\exp (C \abs{s_n}^{-(\tau+1)^{-1}})$, we get from 
\eqref{eq_dir2} that for $n$ sufficiently large $\norm{A_n(t)}\geq 2 \abs{s_n}^{-1}$.
By considering $\{s_j\}$ to decrease fast enough, we can assume that
\begin{equation*}
\abs{\langle k_j, \tilde{\omega}_n \rangle}=\abs{s_j-s_n}\abs{k_{j,1}}>\abs{s_n}, \quad 2\leq j<n,
\end{equation*}
 and so $\norm{B_n(t)}<1$ as well. We conclude that 
$ \norm{\tilde{r}(t)}>\abs{s_n}^{-1}= \norm{z}^{-1}$,
which implies that $\Phi_n$ satisfies ${(\mathcal{P}^2_n)}$. 

Consider case $iii)$. We define $z$ as in cases $i)$ and $ii)$ but with $s_n:=e^{-n^2 \norm{k_n}}$. The flow becomes  
$\Phi_{n}^t(z)=(\tilde{\omega} t, \theta_d(t), \tilde{r}(t), s_n)$ with
\begin{equation*}
\label{eq_r_flow3}
 \tilde{r}(t)=A_n(t)+ B_n(t),
\end{equation*}
where now
\begin{equation*}
\label{eq_A_B3}
 \begin{array}{l}
   A_n(t)= s_n^n \cfrac{e^{-n\norm{k_n}}}{\langle k_n, \tilde{\omega} \rangle} k_n \sin(2\pi \langle k_n, \tilde{\omega} \rangle t)   ,\\ 
   B_n(t)=\sum_{2\leq j< n}{s_n^j \cfrac{ e^{-j\norm{k_j}}  }{\langle k_j, \tilde{\omega} \rangle} k_j \sin(2\pi \langle k_j, \tilde{\omega} \rangle t)}.
 \end{array}
\end{equation*}

\vspace{0.2cm}

We assume as before that $\norm{k_n}$ grows sufficiently fast to guarantee that $\norm{B_n(t)}<1$. 
Also due to \eqref{eq_liouville_cond} we can assume that $k_n$ are chosen in such a way that $\abs{\langle \tilde{\omega}, k_n \rangle} \leq e^{-n^4{\norm{k_n}}}$. 
Hence for $t:=s_n^{-2n}=e^{2n^3{\norm{k_n}}}$ and $n$ big enough, $|\sin(2\pi \langle k_n, \tilde{\omega} \rangle t)|> |\langle k_n, \tilde{\omega} \rangle| t$. Therefore $\norm{\tilde r(t)}>s_n^{-1}$, which implies that $\Phi_n$ satisfies { $(\mathcal{P}_{n}^3)$}.
\medskip

The proof of case $iv)$ follows exactly in the same way, with the same choice of $s_n$ and the initial condition $z$,  
but with $s_n^2$ in place of $s_n^n$ and $s_n^j$ in the expressions of $A_n$ and $B_n$, which allows to take $t:=s_n^{-4}$ and obtain 
{$(\mathcal{P}_{n}^4)$} instead of {$(\mathcal{P}_{n}^3)$}.

\medskip

We consider now case $v)$. For any initial condition $z=(\theta, r)$, the flow of $H_n$ satisfies $\Phi_{n}^{t}(z)=(\tilde{\theta}+\tilde{\omega} t, \theta_{d}(t), \tilde{r}(t), s)$, where 
 \begin{equation*}
  \tilde{r}(t)=\tilde{r}+ \sum_{j=2}^{n}{k_j {s^j} e^{\norm{k_j} s} \sin(2\pi \langle k_j, \tilde{\theta}+ \tilde{\omega} t \rangle)}-
  \sum_{j=2}^{n}{k_j {s^j} e^{\norm{k_j} s} \sin(2\pi \langle k_j, \tilde{\theta} \rangle)}.
 \end{equation*}
 Notice that if we define $\tau_n:=\abs{\langle k_n, \tilde{\omega} \rangle}^{-1}$, then there exists $0<t< \tau_n$ such that $\abs{\sin(2\pi \langle k_n, \tilde{\theta}+\tilde{\omega} t \rangle)- \sin(2\pi \langle k_n, \tilde{\theta}\rangle)}=1$.
 Therefore by choosing $\{k_n\}$ increasing sufficiently fast in norm we can impose that for all $z \in Q_n^+$ there exits $t>0$ such that  
\begin{equation*}
 \norm{\tilde{r}(t)} \geq \bigg{|} \frac{\norm{k_n}}{n^j}e^{\norm{k_n}/n}-\sum_{j=2}^{n-1}{n^j \norm{k_j} e^{n \norm{k_j}}} -n\bigg{|}>n.
\end{equation*}
Case $vii)$ is similar to case $v)$ except for the fact that the exponent $\norm{k_j} s^2$ is positive yields diffusion on all $Q_n$ instead of $Q_n^+$. \end{proof} \bk

\subsection{Concluding the proofs} We can now finish the proofs of Theorems \ref{th1}--\ref{th4}. 

\begin{proofof}{Theorem \ref{th1}} The convergence of $H_n$ was proved in Proposition \ref{prop.convergence}. The characterization of the BNF was proved in Proposition \ref{prop.BNF}. The instability comes from the fact that the flow of $H$ satisfies $(\mathcal{P}_{\infty}^1)$, which follows from Propositions \ref{prop.approximation} and \ref{prop.diffusion} (provided the sequence $\{k_j\}$ is chosen to grow sufficiently fast).
It is left to verify that if $\omega(s)=\bar{\omega}(s)$ and $\omega$ is Diophantine then $\cT_0$ is KAM stable. From Theorem \ref{mB}, it suffices  to see that $\bar{N}$ satisfies the R\"ussmann non-degeneracy condition,
namely that there does not exist any vector $\gamma\neq 0$ such that for every $r$ in some neighbourhood of $\cT_0$

\begin{equation}
\label{eq_russman}
\langle \nabla\bar{N}(r),\gamma \rangle=0.
\end{equation}
In our case we have 
\begin{equation*}
\nabla \bar N(r)=(\omega_1+s, \ldots, \omega_{d-1}+s^{d-1}, \omega_d+\sum_{l=1}^{d-1}{r_l ls^{l-1}} ),
\end{equation*}
and it is readily seen that \eqref{eq_russman} forces $\gamma$ to be zero. 
Hence $\bar{N}$ is R\"ussmann non-degenerate. 
This finishes the proof of Theorem \ref{th1}.\end{proofof} \bk 

\begin{proofof}{Theorem \ref{th2}}   The convergence of $H_n$ was proved in Proposition \ref{prop.convergence}. 
The characterization of the BNF was proved in Proposition \ref{prop.BNF}. 
The upper bound on the diffusion times comes from the fact that the flow of $H$ satisfies { $(\mathcal{P}_{\infty}^2)$}, 
which follows from Propositions \ref{prop.approximation}  and \ref{prop.diffusion} (provided the sequence $\{k_j\}$ is chosen to grow sufficiently fast).
\end{proofof} \bk 

\begin{proofof}{Theorem \ref{th03}}  The convergence of $H_n$ was proved in Proposition \ref{prop.convergence}. 
The characterization of the BNF for part {a)} was proved in Proposition \ref{prop.BNF}. 
The estimate on the diffusion times comes from the flow of $H$ satisfying {$(\mathcal{P}_{\infty}^3)$} or {$(\mathcal{P}_{\infty}^4)$}, 
that follow as in the proof of Theorem \ref{th2} from Propositions \ref{prop.approximation} and \ref{prop.diffusion}. 
\end{proofof} \bk

\begin{proofof}{Theorem \ref{th3}}  The convergence of $H_n$ was proved in Proposition \ref{prop.convergence}. 
The characterization of the BNF was proved in Proposition \ref{prop.BNF}. The diffusion for $r_d>0$ comes from the flow of $H$ satisfying {$(\mathcal{P}_{\infty}^5)$}, which holds again due to Propositions \ref{prop.approximation} and \ref{prop.diffusion}. 
The integrability on $M^-$ was proved in Proposition \ref{prop.conjugacy}. 
\end{proofof} \bk 

\begin{proofof}{Theorem \ref{th3bis}}  The convergence of $H_n$ was proved in Proposition \ref{prop.convergence}. 
The characterization of the BNF was proved in Proposition \ref{prop.BNF}. The $C^{l}$ and not $C^{l+1}$ integrability was proved  in Proposition \ref{prop.conjugacy}. 
\end{proofof} \bk 

\begin{proofof}{Theorem \ref{th4}}  The convergence of $H_n$ was proved in Proposition \ref{prop.convergence}. 
The characterization of the BNF was proved in Proposition \ref{prop.BNF}. The diffusion for $r_d\neq 0$ comes from the flow of $H$ satisfying {$(\mathcal{P}_{\infty}^6)$},  again by Propositions \ref{prop.approximation} and \ref{prop.diffusion}.
\end{proofof} \bk

\bigskip

\noindent \textbf{\bg Acknowledgements.} We would like to thank Abed Bounemoura, Håkan Eliasson, Rapha\"el Krikorian, Maria Saprykina and Mikhail Sevryuk for many useful discussions all along this work. We are also very grateful to the referee for many useful suggestions and comments.
This work has been supported by the Swedish Research Council (VR 2015-04012) and the Knut and Alice Wallenberg foundation (KAW 2016.0403).

\end{document}